\documentclass[10pt,twoside,reqno]{amsart}
\usepackage{amsmath}
\usepackage{amssymb}
\usepackage{amsfonts}
\usepackage{amsthm}
\usepackage{cite}
\newtheorem{theorem}{Theorem}[section]
\newtheorem{corollary}[theorem]{Corollary}
\newtheorem{lemma}[theorem]{Lemma}
\newtheorem{proposition}[theorem]{Proposition}
\theoremstyle{definition}
\newtheorem{definition}[theorem]{Definition}
\newtheorem{remark}[theorem]{Remark}
\newtheorem{example}[theorem]{Example}
\numberwithin{equation}{section}

\newcommand{\bc}{\begin{center}}
\newcommand{\ec}{\end{center}}

\begin{document}


\title[The Method of Monotone Iterations for Mixed Monotone Operators]%
{\large The Method of Monotone Iterations for Mixed Monotone Operators
in Partially Ordered Sets and Order-Attractive Fixed Points}
\author[]{Mircea--Dan Rus}
\date{\ }
\maketitle

\vspace*{-0.5cm}

\begin{center}
{\footnotesize
Department of Mathematics, Technical University of Cluj-Napoca\\
Str. Memorandumului nr. 28, 400114 Cluj-Napoca, Romania\\
E-mail: rus.mircea@math.utcluj.ro}
\end{center}

\bigskip

{\footnotesize
\noindent
{\bf Abstract.}
We use the method of monotone iterations to obtain fixed point and coupled
fixed point results for mixed monotone operators in the setting
of partially ordered sets, with no additional assumptions on the partial order
and with no convergence structure. We define the concept of attractive
fixed point with respect to the partial order and obtain several criteria for
the existence, uniqueness and order-attractiveness of the fixed points, both
in the presence and in the absence of a coupled lower-upper fixed point.
As an application, we present a fixed point result for a class of mixed monotone
operators in the setting of ordered linear spaces.


}

\bigskip

\section{Introduction and preliminaries}

A fundamental principle both in mathematics and computer science is iteration.
Particularly, fixed point iteration and monotone iterative techniques are the
core methods when solving a large class of abstract and applied mathematical
problems and play an important part in many algorithms.

Monotone iterative methods (in connection with the method of lower and upper
solutions) go back at least to E. Picard
\cite{Picard1890,Picard1893,Picard1900} in the 1890s, in the study of the
Dirichlet problem for nonlinear second order (ordinary and partial)
differential equations. Since then, these methods have been further developed
in more abstract settings and have been used to solve a wide variety of
nonlinear problems arising from various fields of science. In this direction,
the class of operators to which these methods were applied has been enlarged
to include operators with more general monotonicity-type properties, like the
mixed-monotone property.

In this context, most of the abstract fixed point results for the class of
mixed monotone operators that make use of monotone iterative techniques were
formulated in the framework of ordered topological spaces (particularly,
ordered Banach spaces) (e.g.,
\cite{Opoitsev1975,Opoitsev1978a,Moore1981,Guo1987,Guo1988,Guo2004}),
partially ordered metric spaces (e.g.,
\cite{GnanaBhaskar2006,Lakshmikantham2009,Samet2010,Harjani2011a,Luong2011,Rus2011}%
) and partially ordered cone metric spaces (e.g.,
\cite{Sabetghadam2009,Karapinar2010,Shatanawi2010,Shatanawi2010a,Olatinwo2011}%
). This seems perfectly justified by the need of some convergence structure
that is compatible in some way with the partial order, such that one can
consistently describe the result of the iterative process.

Following this long line of research, both pure and applied, the aim of this
paper is to show that it is still possible to obtain constructive fixed point
results by monotone iteration without assuming any convergence structure, in
the setting of partially ordered sets and with no additional assumptions on
the partial order. In particular, we are interested in obtaining criteria for
the existence, uniqueness and attractiveness (in some predefined sense) of the
fixed points, exclusively by means of explicit iterative techniques, both in
the presence and in the absence of a coupled lower-upper fixed point. Also, we
choose to study the class of mixed monotone operators since it contains both
the classes of nondecreasing and nonincreasing operators, respectively, in one
unified approach, while being large enough to describe a great number of
nonlinear problems where usual monotonicity is not present.

Recall that if $(X,\leq)$ is a partially ordered set and $A:X\times
X\rightarrow X$, then $A$ is said to be \emph{mixed monotone} (or is said to
have \emph{the mixed monotone property)} if $A$ is nondecreasing in the first
argument and nonincreasing in the second argument, i.e.,%
\[
x_{1},x_{2},y_{1},y_{2}\in X,\quad x_{1}\leq x_{2},~y_{1}\geq y_{2}\Rightarrow
A(x_{1},y_{1})\leq A(x_{2},y_{2}).
\]
A pair $(x,y)\in X\times X$ is called \emph{a coupled fixed point of} $A$ if
\[
A(x,y)=x,\text{\quad}A(y,x)=y,
\]
and it is called \emph{a coupled lower-upper fixed point of} $A$ if
\[
x\leq y,\quad x\leq A(x,y),\quad y\geq A(y,x).
\]
Also, $x\in X$ is called \emph{a fixed point} of $A$ if $A(x,x)=x$, i.e.,
$(x,x)$ is a coupled fixed point of $A$. For more details, we refer to
\cite{Guo1987,GnanaBhaskar2006}.

\begin{remark}
While the term \textquotedblleft\emph{mixed monotone}\textquotedblright\ is
due to Lakshmikantam and Guo \cite{Guo1987}, the concept of mixed monotone
operator and the corresponding iterative method go back at least to
Kurpel$^{\prime}$ \cite{Kurpel1968} in the study of two-sided operator
inequalities and their applications to approximating the solutions of
integral, differential, integro-differential and finite (algebraic and
transcendental) equations. We point in this direction to the monograph of
Kurpel$^{\prime}$ and {\v{S}}uvar \cite{Kurpel1980}. Later on, Opo{\u{\i}}tsev
\cite{Opoitsev1975,Opoitsev1978a} established the first (to the best of our
knowledge) fixed point and coupled fixed point results for this type of
operators, in the framework of ordered Banach spaces. In the past three
decades, the results of Opo{\u{\i}}tsev have been rediscovered\ in various
forms and have been extended by many authors (we refer to \cite{Guo2004} for
an overview of the results published on this topic since the 1980s).
Regrettably, none of them seems to have been aware of the results of
Opo{\u{\i}}tsev, although English translations of his works have been
available right after their initial publication in Russian.
\end{remark}

In what follows, we will make use of the following notions and notations.

Let $(X,\leq)$ be a partially ordered set. If $x,y\in X$ are such that $x\leq
y$, then $[x,y]$ denotes the set of all elements $z\in X$ such that $x\leq
z\leq y$. Also, if $(u_{n})$ is a sequence in $X$, then $\sup u_{n}$ and $\inf
u_{n}$ denote the supremum, i.e., the least upper bound, and the infimum,
i.e., the greatest lower bound (when they exist), respectively, of the set
$\{u_{n}:n\in\mathbb{N\}}$, where $\mathbb{N}$ represents the set of all
nonnegative integers. We also write $\sup_{n\geq k}u_{n}$ and $\inf_{n\geq
k}u_{n}$ (for any $k\in\mathbb{N}$) to denote the supremum and the infimum,
respectively, of the set $\{u_{n}:n\geq k\mathbb{\}}$ .

In order to properly define the iterates of any bivariate operator, we need a
composition rule that applies to this class of mappings, hence for any
operators $A,B:X\times X\rightarrow X$ define (cf. \cite{Rus2011}) \emph{the
symmetric composition} (or, \emph{the }$s$\emph{-composition} for short) of
$A$ and $B$ by%
\[
B\ast A:X\times X\rightarrow X,\quad(B\ast A)(x,y)=B(A(x,y),A(y,x))\quad
(x,y\in X).
\]
The $s$-composition is associative and the canonical projection
\[
P_{X}:X\times X\rightarrow X,\quad P(x,y)=x\quad(x,y\in X)
\]
is the identity element, hence one can define the functional powers (i.e., the
iterates) of any operator $A:X\times X\rightarrow X$ with respect to the
$s$-composition by%
\[
A^{n+1}=A\ast A^{n}=A^{n}\ast A\quad(n=0,1,...),\quad A^{0}=P_{X}\text{.}%
\]
When $(X,\leq)$ is a partially ordered set, the $s$-composition of mixed
monotone operators has also the mixed monotone property, hence the iterates of
a mixed monotone operator are also mixed monotone. For more details on this
topic, we refer to \cite{Rus2011}.

\section{Main results}

From this point forward in this Section, it will be assumed that $(X,\leq)$ is
a partially ordered set, $A:X\times X\rightarrow X$ is a mixed monotone
operator and $x_{0},y_{0}\in X$ are such that $x_{0}\leq y_{0}$. Also, define
the sequences $\left(  x_{n}\right)  $ and $\left(  y_{n}\right)  $
recursively by%
\begin{equation}
x_{n+1}=A\left(  x_{n},y_{n}\right)  ,\quad y_{n+1}=A\left(  y_{n}%
,x_{n}\right)  \qquad(n\in\mathbb{N}), \label{eq:0711_00}%
\end{equation}
or, equivalently, by%
\[
x_{n}=A^{n}(x_{0},y_{0}),\quad y_{n}=A^{n}(y_{0},x_{0})\qquad(n\in
\mathbb{N)}.
\]
This coupled iteration together with the results contained in the following
lemma represent the core of the method of monotone iterations for mixed
monotone operators. These ideas are not new and can be found spread throughout
the entire literature that studies the (coupled) fixed points for mixed
monotone operators, though they are usually considered in a less general
setting and are sometimes \emph{hidden} inside proofs. Note that the
assumption of $(x_{0},y_{0})$ being a coupled lower-upper fixed point of $A$
is not essential for obtaining most of the (coupled) fixed point results in
this paper, hence it will be considered as a separate assumption, which
represents a new approach.

\begin{lemma}
\label{th:0711_01}The following properties take place:

\begin{enumerate}
\item For all $n\in\mathbb{N}$, $x_{n}\leq y_{n}$ and
\begin{equation}
x,y\in\lbrack x_{n},y_{n}]\Rightarrow A\left(  x,y\right)  \in\left[
x_{n+1},y_{n+1}\right]  . \label{eq:0711_01}%
\end{equation}

\item If $(x,y)\in\lbrack x_{0},y_{0}]\times\lbrack x_{0},y_{0}]$ is a coupled
fixed point of $A$, then
\[
x,y\in\bigcap\limits_{n\geq0}\left[  x_{n},y_{n}\right]  .
\]

\item If $(x_{0},y_{0})$ is a coupled lower-upper fixed point of $A$, then
$(x_{n})$ is nondecreasing, $(y_{n})$ is nonincreasing and $(x_{n},y_{n})$ is
a coupled lower-upper fixed point of $A$, for all $n\in\mathbb{N}$.
\end{enumerate}
\end{lemma}

\begin{proof}
~

\begin{enumerate}
\item The proof is by induction on $n$. Assume that $x_{n}\leq y_{n}$ for some
$n\in\mathbb{N}$ and consider arbitrary $x,y\in\lbrack x_{n},y_{n}]$. By the
mixed monotonicity of $A$,
\[
x_{n+1}=A(x_{n},y_{n})\leq A(x,y)\leq A\left(  y_{n},x_{n}\right)  =y_{n+1},
\]
hence $x_{n+1}\leq y_{n+1}$ and $A(x,y)\in\lbrack x_{n+1},y_{n+1}]$, which
proves (\ref{eq:0711_01}). Since our assumption is true for $n=0$, the proof
of \textbf{1} is complete.

\item Since $A(x,y)=x$, $A(y,x)=y$ and $x,y\in\lbrack x_{0},y_{0}]$, it
follows that $x,y\in\lbrack x_{n},y_{n}]$ for all $n\in\mathbb{N}$ as a direct
consequence of (\ref{eq:0711_01}), by induction on $n$.

\item Assume that $x_{n}\leq x_{n+1}$ and $y_{n}\geq y_{n+1}$ for some
$n\in\mathbb{N}$. Note that this is equivalent to $(x_{n},y_{n})$ being a
coupled lower-upper fixed point of $A$, since $x_{n}\leq y_{n}$ by \textbf{1}.
Then,%
\begin{align*}
x_{n+1}  &  =A(x_{n},y_{n})\leq A(x_{n+1},y_{n+1})=x_{n+2}\\
y_{n+1}  &  =A(y_{n},x_{n})\geq A(y_{n+1},x_{n+1})=y_{n+2}.
\end{align*}
Since our assumption is true for $n=0$, it follows by induction that
$x_{n}\leq x_{n+1}$ and $y_{n}\geq y_{n+1}$ (hence $(x_{n},y_{n})$ is a
coupled lower-upper fixed point of $A$) for all $n\in\mathbb{N}$.
\end{enumerate}
\end{proof}

The following result is a direct consequence of Lemma \ref{th:0711_01} and
provides a negative answer on the existence of (coupled) fixed points.

\begin{corollary}
\label{th:0711_02}If $\bigcap\limits_{n\geq0}\left[  x_{n},y_{n}\right]
=\emptyset$, then $A$ has no coupled fixed points in $\left[  x_{0}%
,y_{0}\right]  \times\lbrack x_{0},y_{0}]$ (hence, no fixed points in
$[x_{0},y_{0}] $).
\end{corollary}

\subsection{Order-attractive points for mixed monotone operators}

Before we formulate and prove the main fixed point theorems, we need to
introduce and study some new notions.

\begin{definition}
\label{def:0711_01}A point $x^{\ast}\in X$ is said to be $(x_{0},y_{0}%
)$-\emph{weakly order-attractive for }$A$ if $\bigcap\limits_{n\geq0}\left[
x_{n},y_{n}\right]  =\{x^{\ast}\}$, and we denote this by $(x_{0}%
,y_{0})\overset{A}{\rightarrow}x^{\ast}$. Alternatively, we may say that
$x^{\ast}$ \emph{weakly order-attracts} $(x_{0},y_{0})$ \emph{through }$A$, or
that $(x_{0},y_{0})$ \emph{is weakly order-attracted by }$x^{\ast}$
\emph{through }$A$.
\end{definition}

\begin{definition}
\label{def:0711_02}A point $x^{\ast}\in X$ is said to be $(x_{0},y_{0}%
)$-\emph{order-attractive for }$A$ if $\sup x_{n}=\inf y_{n}=x^{\ast}$, and we
denote this by $(x_{0},y_{0})\overset{A}{\rightrightarrows}x^{\ast}$.
Alternatively, we may say that $x^{\ast}$ \emph{order-attracts} $(x_{0}%
,y_{0})$ \emph{through }$A$, or that $(x_{0},y_{0})$ \emph{is order-attracted
by }$x^{\ast}$ \emph{through }$A$.
\end{definition}

\begin{definition}
\label{def:0711_03}A point $x^{\ast}\in X$ is said to be \emph{weakly
order-attractive for }$A$\emph{\ on }$[x_{0},y_{0}]$ if $x^{\ast}\in\lbrack
x_{0},y_{0}]$ and $(u_{0},v_{0})\overset{A}{\rightarrow}x^{\ast}$ for all
$u_{0},v_{0}\in\lbrack x_{0},y_{0}]$ with $u_{0}\leq x^{\ast}\leq v_{0}$, and
we denote this by $[x_{0},y_{0}]\overset{A}{\rightarrow}x^{\ast}$.
Alternatively, we may say that $x^{\ast}$\emph{\ weakly order-attracts}
$[x_{0},y_{0}]$ \emph{through }$A$, or that $[x_{0},y_{0}]$ \emph{is weakly
order-attracted by }$x^{\ast}$ \emph{through }$A$.
\end{definition}

\begin{definition}
\label{def:0711_04}A point $x^{\ast}\in X$ is said to be
\emph{order-attractive for }$A$\emph{\ on }$[x_{0},y_{0}]$ if $x^{\ast}%
\in\lbrack x_{0},y_{0}]$ and $(u_{0},v_{0})\overset{A}{\rightrightarrows
}x^{\ast}$ for all $u_{0},v_{0}\in\lbrack x_{0},y_{0}]$ with $u_{0}\leq
x^{\ast}\leq v_{0}$, and we denote this by $[x_{0},y_{0}%
]\overset{A}{\rightrightarrows}x^{\ast}$. Alternatively, we may say that
$x^{\ast}$ \emph{order-attracts} $[x_{0},y_{0}]$ \emph{through }$A$, or that
$[x_{0},y_{0}]$ \emph{is order-attracted by }$x^{\ast}$ \emph{through }$A$.
\end{definition}

\begin{proposition}
\label{th:0711_03}Let $x^{\ast}\in X$. The following properties take place:

\begin{enumerate}
\item If $(x_{0},y_{0})\overset{A}{\rightarrow}x^{\ast}$, then $x^{\ast}%
\in\lbrack x_{0},y_{0}]$.

\item If $[x_{0},y_{0}]\overset{A}{\rightarrow}x^{\ast}$, then $(x_{0}%
,y_{0})\overset{A}{\rightarrow}x^{\ast}$ and $[u_{0},v_{0}%
]\overset{A}{\rightarrow}x^{\ast}$ for all $u_{0},v_{0}\in\lbrack x_{0}%
,y_{0}]$ with $u_{0}\leq x^{\ast}\leq v_{0}$.

\item If $[x_{0},y_{0}]\overset{A}{\rightrightarrows}x^{\ast}$, then
$(x_{0},y_{0})\overset{A}{\rightrightarrows}x^{\ast}$ and $[u_{0}%
,v_{0}]\overset{A}{\rightrightarrows}x^{\ast}$ for all $u_{0},v_{0}\in\lbrack
x_{0},y_{0}]$ with $u_{0}\leq x^{\ast}\leq v_{0}$.

\item $(x_{0},y_{0})\overset{A}{\rightrightarrows}x^{\ast}$ if and only
if\emph{\ }$(x_{0},y_{0})\overset{A}{\rightarrow}x^{\ast}$ and $\sup
x_{n},\inf y_{n}$ exist.

\item If $[x_{0},y_{0}]\overset{A}{\rightrightarrows}x^{\ast}$, then
$[x_{0},y_{0}]\overset{A}{\rightarrow}x^{\ast}$.
\end{enumerate}
\end{proposition}

\begin{proof}
\textbf{1}, \textbf{2} and \textbf{3} are direct consequences of the
definitions. Also, \textbf{5} follows from \textbf{4} and the definitions,
hence we only need to prove \textbf{4}.

If $(x_{0},y_{0})\overset{A}{\rightrightarrows}x^{\ast}$, then $\sup
x_{n},\inf y_{n}~$exist and $\sup x_{n}=\inf y_{n}=x^{\ast}$, hence $x_{n}\leq
x^{\ast}\leq y_{n}$ for all $n\in\mathbb{N}$. Now, let $x\in X$ such that
$x_{n}\leq x\leq y_{n}$ for all $n\in\mathbb{N}$. Then%
\[
x^{\ast}=\sup x_{n}\leq x\leq\inf y_{n}=x^{\ast},
\]
hence $x^{\ast}=x$. Concluding, $\bigcap\limits_{n\geq0}\left[  x_{n}%
,y_{n}\right]  =\{x^{\ast}\}$, i.e., $(x_{0},y_{0})\overset{A}{\rightarrow
}x^{\ast}$.

Conversely, if $(x_{0},y_{0})\overset{A}{\rightarrow}x^{\ast}$, then
$x_{n}\leq x^{\ast}\leq y_{n}$ for all $n\in\mathbb{N}$ and since $\sup
x_{n},\inf y_{n}$ exist, it follows that $x_{n}\leq\sup x_{n}\leq x^{\ast}%
\leq\inf y_{n}\leq y_{n}$ for all $n\in\mathbb{N}$, hence%
\[
\sup x_{n},\inf y_{n}\in\bigcap\limits_{n\geq0}\left[  x_{n},y_{n}\right]
=\{x^{\ast}\},
\]
which proves that $(x_{0},y_{0})\overset{A}{\rightrightarrows}x^{\ast}$.
\end{proof}

In the following result we establish the properties of (weakly)
ordered-attractive fixed points.

\begin{theorem}
\label{th:0711_06}Let $x^{\ast}\in X$. The following equivalences take place:

\begin{enumerate}
\item $[x_{0},y_{0}]\overset{A}{\rightarrow}x^{\ast}$ if and only if
$(x_{0},y_{0})\overset{A}{\rightarrow}x^{\ast}$ and $x^{\ast}$ is a fixed
point of $A$.

\item $[x_{0},y_{0}]\overset{A}{\rightrightarrows}x^{\ast}$ if and only if
$(x_{0},y_{0})\overset{A}{\rightrightarrows}x^{\ast}$ and $x^{\ast}$ is a
fixed point of $A$.
\end{enumerate}

Moreover, in any of the above situations, $(x^{\ast},x^{\ast})$ is the unique
coupled fixed point of $A$ in $[x_{0},y_{0}]\times\lbrack x_{0},y_{0}] $
(hence, $x^{\ast}$ is the unique fixed point of $A$ in $[x_{0},y_{0}]$).
\end{theorem}

\begin{proof}
First, we prove the direct implications.

Assume that $[x_{0},y_{0}]\overset{A}{\rightarrow}x^{\ast}$. Then $x^{\ast}%
\in\lbrack x_{0},y_{0}]$, hence $(x^{\ast},x^{\ast})\overset{A}{\rightarrow
}x^{\ast}$, which ensures that $x^{\ast}$ is a fixed point of $A$. Also,
$(x_{0},y_{0})\overset{A}{\rightarrow}x^{\ast}$ by Proposition
\ref{th:0711_03} and the direct implication in \textbf{1} is proved.

Similarly, if $[x_{0},y_{0}]\overset{A}{\rightrightarrows}x^{\ast}$, then
$(x_{0},y_{0})\overset{A}{\rightrightarrows}x^{\ast}$ and $[x_{0}%
,y_{0}]\overset{A}{\rightarrow}x^{\ast}$ by Proposition \ref{th:0711_03},
hence $x^{\ast}$ is a fixed point of $A$ using the direct implication in
\textbf{1}, and the direct implication in \textbf{2} is also proved.

Now we prove the converse implications.

Assume that $x^{\ast}$ is a fixed point of $A$ and $(x_{0},y_{0}%
)\overset{A}{\rightarrow}x^{\ast}$. We prove that $[x_{0},y_{0}%
]\overset{A}{\rightarrow}x^{\ast}$.

Let $u_{0},v_{0}\in\lbrack x_{0},y_{0}]$ such that $u_{0}\leq x^{\ast}\leq
v_{0}$ and define the sequences $(u_{n}),(v_{n})$ by
\[
u_{n+1}=A(u_{n},v_{n}),~v_{n+1}=A(v_{n},u_{n})\quad(n\in\mathbb{N})
\]
or, equivalently, by%
\[
u_{n}=A^{n}(u_{0},v_{0}),~v_{n}=A^{n}(v_{0},u_{0})\quad(n\in\mathbb{N}).
\]
Since $x^{\ast}$ is a fixed point of $A$, it follows that $x^{\ast}$ is a
fixed point of $A^{n}$ for all $n\in\mathbb{N}$. Also, $A^{n}$ is mixed
monotone for all $n\in\mathbb{N}$, and since
\[
x_{0}\leq u_{0}\leq x^{\ast}\leq v_{0}\leq y_{0},
\]
it follows that
\begin{equation}
x_{n}\leq u_{n}\leq x^{\ast}\leq v_{n}\leq y_{n}\quad\text{for all }%
n\in\mathbb{N}, \label{eq:0711_02}%
\end{equation}
which ensures that%
\[
\{x^{\ast}\}\subseteq\bigcap\limits_{n\geq0}\left[  u_{n},v_{n}\right]
\subseteq\bigcap\limits_{n\geq0}\left[  x_{n},y_{n}\right]  =\{x^{\ast}\},
\]
hence $(u_{0},v_{0})\overset{A}{\rightarrow}x^{\ast}$ and the converse
implication in \textbf{1} is proved.

Now, assume that $x^{\ast}$ is a fixed point of $A$ and $(x_{0},y_{0}%
)\overset{A}{\rightrightarrows}x^{\ast}$. We prove that $[x_{0},y_{0}%
]\overset{A}{\rightrightarrows}x^{\ast}$.

Let $u_{0},v_{0}\in\lbrack x_{0},y_{0}]$ such that $u_{0}\leq x^{\ast}\leq
v_{0}$ and let $(u_{n}),(v_{n})$ as previously defined. By using the same
argument as before, we obtain (\ref{eq:0711_02}), and since $\sup x_{n}=\inf
y_{n}=x^{\ast}$, it follows that $\sup u_{n}=\inf v_{n}=x^{\ast}$, i.e.,
$(u_{0},v_{0})\overset{A}{\rightrightarrows}x^{\ast}$, hence the converse
implication in \textbf{2} is proved.

Finally, we only need to prove that if $x^{\ast}$ is a fixed point of $A$ and
$(x_{0},y_{0})\overset{A}{\rightarrow}x^{\ast}$ (or $(x_{0},y_{0}%
)\overset{A}{\rightrightarrows}x^{\ast}$), then $(x^{\ast},x^{\ast})$ is the
unique coupled fixed point of $A$ in $[x_{0},y_{0}]\times\lbrack x_{0},y_{0}%
]$. Indeed, if $(x,y)\in$ $[x_{0},y_{0}]\times\lbrack x_{0},y_{0}]$ is a
coupled fixed point of $A$, then, by Lemma \ref{th:0711_01}, Definition
\ref{def:0711_01} (and Proposition \ref{th:0711_03}(4)), we have that%
\[
x,y\in\bigcap\limits_{n\geq0}\left[  x_{n},y_{n}\right]  =\{x^{\ast}\},
\]
hence $x=y=x^{\ast}$. Clearly, $(x^{\ast},x^{\ast})$ is a coupled fixed point
of $A$ in $[x_{0},y_{0}]\times\lbrack x_{0},y_{0}]$, and the proof is now complete.
\end{proof}

\begin{remark}
\label{th:0711_04}In general, if $(x_{0},y_{0})\overset{A}{\rightrightarrows
}x^{\ast}$ (or $(x_{0},y_{0})\overset{A}{\rightarrow}x^{\ast}$), then
$x^{\ast}$ is not necessarily a fixed point of $A$ (though, under additional
assumptions, this may be sufficient -- see Theorem \ref{th:0711_08}). The
following elementary example proves this claim by means of a mixed monotone
mapping with no (coupled) fixed points that has a $(x_{0},y_{0})$%
-order-attractive point.
\end{remark}

\begin{example}
Let $A:\mathbb{R}^{2}\rightarrow\mathbb{R}$ by defined by $A(x,y)=x+\dfrac
{1-\{x\}}{2}$, where $\{x\}$ denotes the fractional part of the real number
$x$. Then $A$ is mixed monotone and $(0,1)\overset{A}{\rightrightarrows}1$,
yet $A$ has no (coupled) fixed points.

First, we prove that $A$ is mixed monotone, which, in this case, is equivalent
to $A$ being nondecreasing (with respect to $x$). Let $x_{1},x_{2}%
\in\mathbb{R}$ such that $x_{1}\leq x_{2}$ and let $n=x_{2}-\{x_{2}\}$ be the
integer part of $x_{2}$. If $x_{1}\in\lbrack n,n+1)$, then $x_{1}=n+\{x_{1}%
\}$, hence $\{x_{2}\}-\{x_{1}\}=x_{2}-x_{1}$ and%
\[
A(x_{2},y)-A(x_{1},y)=x_{2}-x_{1}-\frac{\{x_{2}\}-\{x_{1}\}}{2}=\frac
{x_{2}-x_{1}}{2}\geq0\text{.}%
\]
Else, $x_{1}<n\leq x_{2}<n+1$, hence
\[
A(x_{2},y)\geq A(n,y)
\]
(from the previous case, by letting $x_{1}:=n$) and
\[
A(n,y)-A(x_{1},y)=n-x_{1}+\frac{\{x_{1}\}}{2}>0,
\]
which proves that $A(x_{2},y)\geq A(x_{1},y)$.

Now, choose $x_{0}=0$ and $y_{0}=1$. It is a simple exercise to show (e.g., by
induction) that the corresponding sequences $\left(  x_{n}\right)  ,\left(
y_{n}\right)  $ defined by (\ref{eq:0711_00}) are%
\[
x_{n}=1-\frac{1}{2^{n}},\quad y_{n}=2-\frac{1}{2^{n}}\qquad(n\in\mathbb{N)},
\]
hence%
\[
\sup x_{n}=\inf y_{n}=1\text{,}%
\]
proving that $x^{\ast}=1$ is $(x_{0},y_{0})$-order-attractive for $A$.

Finally, it can be easily noticed that $A$ has no (coupled) fixed points,
since $A(x,y)=x$ if and only if $\{x\}=1$, which is impossible.
\end{example}

\subsection{Fixed point theorems}

We conclude with the main results. In essence, we prove in each of the
following results that for a point $x^{\ast}\in X$ to be a weakly
ordered-attractive fixed point of $A$, it is sufficient (under additional
assumptions) that $x^{\ast}$ is $(x_{k},y_{k})$-weakly ordered-attractive for
some $k\in\mathbb{N}$. In particular, if there exists $k\in\mathbb{N}$ such
that $\sup_{n\geq k}x_{n}=\inf_{n\geq k}y_{n}=x^{\ast}$, then $x^{\ast}$ is an
ordered-attractive fixed point of $A$. In this way, we establish several
simple criteria for the existence, uniqueness and (weakly)
order-attractiveness of the fixed points of mixed monotone operators.

\begin{theorem}
\label{th:0711_05}Let $k\geq1$ such that $\bigcap\limits_{n=0}^{k-1}%
[x_{n},y_{n}]$ is non-empty, and $x^{\ast}\in\bigcap\limits_{n=0}^{k-1}%
[x_{n},y_{n}]$.

If $(x_{k},y_{k})\overset{A}{\rightarrow}x^{\ast}$, then $(x^{\ast},x^{\ast})
$ is the unique coupled fixed point of $A$ in $[x_{n},y_{n}]\times\lbrack
x_{n},y_{n}]$, $x^{\ast}$ is the unique fixed point of $A$ in $[x_{n},y_{n}]$
and $[x_{n},y_{n}]\overset{A}{\rightarrow}x^{\ast}$ for all $n\in
\{0,1,\ldots,k\}$.

Additionally, if $\sup_{n\geq k}x_{n}$ and $\inf_{n\geq k}y_{n}$ exist, then
$[x_{n},y_{n}]\overset{A}{\rightrightarrows}x^{\ast}$ for all $n\in
\{0,1,\ldots,k\}$.
\end{theorem}

\begin{proof}
For each $n\in\mathbb{N}$, let $X_{n}=\bigcap\limits_{m\geq n}[x_{m},y_{m}]$.
It is clear that $(x_{n},y_{n})\overset{A}{\rightarrow}x^{\ast}$ if and only
if\emph{ }$X_{n}=\{x^{\ast}\}$, hence the hypothesis ensure that
$X_{k}=\{x^{\ast}\}$ and
\[
X_{0}=\left(  \bigcap\limits_{n=0}^{k-1}[x_{n},y_{n}]\right)  \cap
X_{k}=\{x^{\ast}\}\text{.}%
\]
Since, obviously, $X_{0}\subseteq X_{1}\subseteq\ldots\subseteq X_{k}\subseteq
X_{k+1}\subseteq\ldots$, we conclude that%
\[
X_{0}=X_{1}=\ldots=X_{k}=\{x^{\ast}\},
\]
hence $(x_{n},y_{n})\overset{A}{\rightarrow}x^{\ast}$ for all $n\in
\{0,1,\ldots,k\}$.

Since $x^{\ast}\in\lbrack x_{n},y_{n}]$ for all $n\in\mathbb{N}$, it follows
by (\ref{eq:0711_01}) that $A(x^{\ast},x^{\ast})\in\lbrack x_{n+1},y_{n+1}]$
for all $n\in\mathbb{N}$, hence $A(x^{\ast},x^{\ast})\in X_{1}=\{x^{\ast}\}$,
proving that $x^{\ast}$ is a fixed point of $A$.

The conclusion now follows by applying Theorem \ref{th:0711_06}(1) with
$(x_{0},y_{0})$ replaced by $(x_{n},y_{n})$ ($n\in\{0,1,\ldots,k\}$).

Additionally, assume that $\sup_{n\geq k}x_{n}$ and $\inf_{n\geq k}y_{n}$
exist, hence
\[
\sup_{n\geq k}x_{n}=\inf_{n\geq k}x_{n}=x^{\ast}%
\]
by Proposition \ref{th:0711_03}(4), with $(x_{0},y_{0})$ replaced by
$(x_{k},y_{k})$. Since $x_{m}\leq x^{\ast}\leq y_{m}$ for all $m\in\mathbb{N}$
(by $X_{0}=\{x^{\ast}\}$), it follows that $x^{\ast}=\sup_{m\geq n}x_{m}%
=\inf_{m\geq n}y_{m}$ for all $n\in\{0,1,\ldots,k\}$, i.e., $(x_{n}%
,y_{n})\overset{A}{\rightrightarrows}x^{\ast}$ for all $n\in\{0,1,\ldots,k\}$
and the proof is complete by further applying Theorem \ref{th:0711_06}(2) with
$(x_{0},y_{0})$ replaced by $(x_{n},y_{n})$ ($n\in\{0,1,\ldots,k\}$).
\end{proof}

\begin{corollary}
\label{th:0711_07}Let $x^{\ast}\in\lbrack x_{0},y_{0}]$. If $(x_{1}%
,y_{1})\overset{A}{\rightarrow}x^{\ast}$, then $(x^{\ast},x^{\ast})$ is the
unique coupled fixed point of $A$ in $[x_{0},y_{0}]\times\lbrack x_{0}%
,y_{0}]\cup\lbrack x_{1},y_{1}]\times\lbrack x_{1},y_{1}]$, $x^{\ast}$ is the
unique fixed point of $A$ in $[x_{0},y_{0}]\cup\lbrack x_{1},y_{1}]$ and
$[x_{0},y_{0}]\overset{A}{\rightarrow}x^{\ast}$, $[x_{1},y_{1}%
]\overset{A}{\rightarrow}x^{\ast}$.

Additionally, if $\sup_{n\geq1}x_{n}$ and $\inf_{n\geq1}y_{n}$ exist, then
$[x_{0},y_{0}]\overset{A}{\rightrightarrows}x^{\ast}$ and $[x_{1}%
,y_{1}]\overset{A}{\rightrightarrows}x^{\ast}$.
\end{corollary}

\begin{proof}
This follows by Theorem \ref{th:0711_06} with $k=1$.
\end{proof}

By assuming that $(x_{0},y_{0})$ is a coupled lower-upper fixed point of $A$,
we obtain the following results.

\begin{theorem}
\label{th:0711_08}Let $x^{\ast}\in X$ and assume that $(x_{0},y_{0})$ is a
coupled lower-upper fixed point of $A$. If $(x_{0},y_{0}%
)\overset{A}{\rightarrow}x^{\ast}$, then $(x^{\ast},x^{\ast})$ is the unique
coupled fixed point of $A$ in $[x_{0},y_{0}]\times\lbrack x_{0},y_{0}]$,
$x^{\ast}$ is the unique fixed point of $A$ in $[x_{0},y_{0}]$ and
$[x_{0},y_{0}]\overset{A}{\rightarrow}x^{\ast}$.

Additionally, if $\sup x_{n}$ and $\inf y_{n}$ exist, then $[x_{0}%
,y_{0}]\overset{A}{\rightrightarrows}x^{\ast}$.
\end{theorem}

\begin{proof}
We use the same notations as in the proof of Theorem \ref{th:0711_05}. Since
$(x_{0},y_{0})$ is a coupled lower-upper fixed point of $A$, it follows by
Lemma \ref{th:0711_01} that
\[
x_{0}\leq x_{1}\leq\ldots\leq x_{n}\leq x_{n+1}\leq\ldots\leq y_{n+1}\leq
y_{n}\leq\ldots\leq y_{1}\leq y_{0},
\]
hence $X_{0}=X_{1}$. Since $(x_{0},y_{0})\overset{A}{\rightarrow}x^{\ast}$, we
conclude that $X_{0}=X_{1}=\{x\}$, hence $(x_{1},y_{1})\overset{A}{\rightarrow
}x^{\ast}$ and $x^{\ast}\in\lbrack x_{0},y_{0}]$. The conclusion now follows
by Corollary \ref{th:0711_07}.

Additionally, if $\sup x_{n}$ and $\inf y_{n}$ exist, then $(x_{0}%
,y_{0})\overset{A}{\rightrightarrows}x^{\ast}$ by Proposition \ref{th:0711_03}%
(4), hence $[x_{0},y_{0}]\overset{A}{\rightrightarrows}x^{\ast}$ by Theorem
\ref{th:0711_06}(2), which concludes the proof.
\end{proof}

\begin{remark}
In the conditions of Theorem \ref{th:0711_08}, $x^{\ast}\in\lbrack x_{n}%
,y_{n}]\subseteq\lbrack x_{0},y_{0}]$ for all $n\in\mathbb{N}$, hence the
conclusion of the theorem already contains that $(x^{\ast},x^{\ast})$ is the
unique coupled fixed point of $A$ in $[x_{n},y_{n}]\times\lbrack x_{n},y_{n}]$
and $[x_{n},y_{n}]\overset{A}{\rightarrow}x^{\ast}$ for all $n\geq1$, without
explicitly stating it.
\end{remark}

In many cases, it is possible that the starting pair of the iterative process
is not a coupled lower-upper fixed point, but we arrive to such a pair after
several iterations. This situation is studied next.

\begin{theorem}
\label{th:0711_09}Let $x^{\ast}\in\lbrack x_{0},y_{0}]$ and assume there
exists $k\geq1$ such that $(x_{k},y_{k})$ is a coupled lower-upper fixed point
of $A$. If $(x_{k},y_{k})\overset{A}{\rightarrow}x^{\ast}$, then $(x^{\ast
},x^{\ast})$ is the unique coupled fixed point of $A$ in $[x_{n},y_{n}%
]\times\lbrack x_{n},y_{n}]$, $x^{\ast}$ is the unique fixed point of $A$ in
$[x_{n},y_{n}]$ and $[x_{n},y_{n}]\overset{A}{\rightarrow}x^{\ast}$ for all
$n\in\{0,1,\ldots,k\}$.

Additionally, if $\sup_{n\geq k}x_{n}$ and $\inf_{n\geq k}y_{n}$ exist, then
$[x_{n},y_{n}]\overset{A}{\rightrightarrows}x^{\ast}$ for all $n\in
\{0,1,\ldots,k\}$.
\end{theorem}

\begin{proof}
By applying Theorem \ref{th:0711_08}, with $(x_{0},y_{0})$ replaced by
$(x_{k},y_{k})$, it follows that $x^{\ast}$ is a fixed point of $A$ and, since
$x^{\ast}\in\lbrack x_{0},y_{0}]$, it follows by Lemma \ref{th:0711_01}(2)
that $x^{\ast}\in\bigcap\limits_{n\geq0}\left[  x_{n},y_{n}\right]  $, hence
$x^{\ast}\in\bigcap\limits_{n=0}^{k-1}[x_{n},y_{n}]$. The conclusion now
follows by Theorem \ref{th:0711_05}.
\end{proof}

\begin{remark}
In the conditions of Theorem \ref{th:0711_09}, $x^{\ast}\in\lbrack x_{n}%
,y_{n}]\subseteq\lbrack x_{k},y_{k}]$ for all $n\geq k$, hence the conclusion
of the theorem already contains that $(x^{\ast},x^{\ast})$ is the unique
coupled fixed point of $A$ in $[x_{n},y_{n}]\times\lbrack x_{n},y_{n}]$ and
$[x_{n},y_{n}]\overset{A}{\rightarrow}x^{\ast}$ for all $n\geq k+1$.
\end{remark}

\section{Application}

As an application, we present a fixed point result for a class of mixed
monotone operators in the setting of ordered linear spaces. First, recall some
notions and results.

\subsection{Some preliminaries on ordered linear spaces}

Let $(X,K)$ be an ordered linear space over $\mathbb{R}$, i.e., $X$ is a real
linear space and $K\subseteq X$ a cone in $X$ (i.e., a convex set such that
$\lambda K\subseteq K$ for all $\lambda\geq0$ and $K\cap(-K)=\{\theta\}$,
where $\theta$ denotes the zero element in $X$). Then the relation on $X$
defined by $x\leq y\Leftrightarrow y-x\in K$ is a linear order on $X$, i.e.,
an order that satisfies:

\begin{enumerate}
\item[(i)] $x,y,z\in X:x\leq y\Rightarrow x+z\leq y+z$.

\item[(ii)] $x,y\in X,~\lambda\geq0:x\leq y\Rightarrow\lambda x\leq\lambda y$.
\end{enumerate}

It is said that $K$ is \emph{Archimedean} if $x\leq\theta$ whenever there
exists $y\in X$ such that $nx\leq y$ for all $n\in\mathbb{N}$. It is well
known that if $K$ is Archimedean, then for every $x,y\in X$, $\lambda
\in\mathbb{R}$ and every nonincreasing sequence $\left(  \lambda_{n}\right)  $
convergent to $\lambda$:
\[
x\leq\lambda_{n}y~\text{for all~}n\in\mathbb{N}\Rightarrow x\leq\lambda y.
\]

Two elements $x,y$ in $K$ are said to be \emph{linked} (cf.
\cite{Thompson1963}) if there exists $\lambda\in(0,1)$ such that $\lambda
x\leq y$ and $\lambda y\leq x$. This is an equivalence which splits $K$ into
disjoint components (called \emph{parts}).

For further details on these topics we refer to, e.g., \cite{Jameson1968}.

In order to state and prove the main result in this Section, we need to
consider some new notions.

\begin{definition}
A sequence $(x_{n})$ in $X$ is said to be:

\begin{enumerate}
\item[(i)] \emph{upper self-bounded} if for every $\mu>1$ exists
$k\in\mathbb{N}$ such that $x_{n}\leq\mu x_{k}$ for all $n\geq k$;

\item[(ii)] \emph{lower self-bounded} if for every $\lambda\in(0,1)$ exists
$k\in\mathbb{N}$ such that $\lambda x_{k}\leq x_{n}$ for all $n\geq k$.
\end{enumerate}
\end{definition}

\begin{example}
Every nondecreasing sequence in $K$ is lower self-bounded. Similarly, every
nonincreasing sequence in $K$ is upper self-bounded.
\end{example}

\begin{definition}
$K$ is said to be \emph{self-complete} if every nondecreasing sequence in $K$
that is upper self-bounded has supremum.
\end{definition}

\begin{remark}
It is not hard to prove the following equivalence: $K$ is self-complete if and
only if every nonincreasing sequence in $K$ that is lower self-bounded has
infimum. Since this result is not essential in our arguments, we omit its proof.
\end{remark}

\begin{example}
Let $n\in\mathbb{N}$, $n\geq1$ and
\[
\mathbb{R}_{+}^{n}=\left\{  x=(x^{1},x^{2},\ldots,x^{n})\in\mathbb{R}%
^{n}:x^{i}\geq0\text{ for all }i\in\{1,2,\ldots,n\}\right\}
\]
be the nonnegative cone in $\mathbb{R}^{n}$. Then $\mathbb{R}_{+}^{n}$ is
Archimedean and self-complete.

Indeed, $\mathbb{R}_{+}^{n}$ is Archimedean since for every $x,y\in
\mathbb{R}^{n}$:%
\begin{align*}
nx\leq y\text{ for all }n\in\mathbb{N}  &  \Leftrightarrow x^{i}\leq
\frac{y^{i}}{n}\text{ for all }n\in\mathbb{N}\text{, }i\in\{1,2,\ldots,n\}\\
&  \mathbb{\Rightarrow}x^{i}\leq0\text{ for all }i\in\{1,2,\ldots
,n\}\Leftrightarrow x\leq\theta\text{.}%
\end{align*}
Also, if $(x_{n})$ is a nondecreasing sequence in $\mathbb{R}_{+}^{n}$ that is
upper self-bounded, then for every $i\in\{1,2,\ldots,n\}$ the sequence
$(x_{n}^{i})$ is nondecreasing and bounded (in $\mathbb{R})$, hence has
supremum, which concludes the argument.
\end{example}

\begin{example}
Let $Q$ be a compact Hausdorff topological space and $C(Q)$ be the linear
space of all real valued continuous functions on $Q$, while
\[
K=\left\{  x\in C(Q):x(t)\geq0\text{ for all }t\in Q\right\}
\]
is the cone of all nonnegative functions in $C(Q)$. Then $K$ is Archimedean
and self-complete.

Indeed, $K$ is Archimedean since for every $x,y\in C(Q)$:%
\[
nx\leq y\text{ for all }n\in\mathbb{N}\Leftrightarrow x(t)\leq\frac{y(t)}%
{n}\text{ for all }n\in\mathbb{N}\text{, }t\in Q\Rightarrow x(t)\leq0\text{
for all }t\in Q\text{.}%
\]
Next, let $(x_{n})$ be a nondecreasing sequence in $K$ that is upper
self-bounded and let $k\in\mathbb{N}$ be such that $x_{n}\leq2x_{k}$ for all
$n\geq k$, hence%
\begin{equation}
0\leq x_{n}(t)\leq2x_{k}(t)\leq2M\quad\text{for all }t\in Q\text{ and }%
n\in\mathbb{N}, \label{eq:0711_03}%
\end{equation}
where $M=\sup_{t}x_{k}(t)$. Now, let $x:Q\rightarrow\mathbb{R}$ be given by
\[
x(t)=\sup_{n}x_{n}(t)\quad(t\in Q)\text{.}%
\]
Clearly, $x$ is correctly defined, i.e., $x(t)$ is finite for all $t\in Q$, by
(\ref{eq:0711_03}). In order to show that $x=\sup x_{n}$ (in the ordered
linear space $(C(Q),K)$), we only need to prove that $x$ is continuous.

Let $\varepsilon>0$ and let $n_{\varepsilon}$ be such that%
\[
x_{n}(t)\leq\left(  1+\frac{\varepsilon}{2M}\right)  x_{n_{\varepsilon}%
}(t)\quad\text{for all }t\in Q\text{ and }n\in\mathbb{N}\text{,}%
\]
hence%
\[
x_{n}(t)\leq x(t)\leq\left(  1+\frac{\varepsilon}{2M}\right)  x_{n}%
(t)\quad\text{for all }t\in Q\text{ and }n\geq n_{\varepsilon}\text{,}%
\]
and by using (\ref{eq:0711_03}), we finally obtain that
\[
0\leq x(t)-x_{n}(t)\leq\frac{\varepsilon}{2M}x_{n}(t)\leq\varepsilon
\quad\text{for all }t\in Q\text{ and }n\geq n_{\varepsilon}\text{,}%
\]
which proves that $(x_{n})$ uniformly converges to $x$, hence $x$ is continuous.
\end{example}

\subsection{A fixed point theorem}

We conclude with a result which establishes the existence, uniqueness and
order-attractiveness of fixed points for a class of mixed monotone operators,
in the context of ordered linear spaces endowed with an Archimedean and
self-complete cone. Our result complements and generalizes \cite[Cor.
3.2]{Chen1993}, \cite[Th. 1]{Guo1988}, \cite[Th. 2.9]{Opoitsev1978a},
\cite[Th. 2.1]{Wu2006}, \cite[Th. 1]{Xu2005a}.

\begin{theorem}
Let $(X,K)$ be an ordered linear space over $\mathbb{R}$ such that $K$ is
Archimedean and self-complete. Let $P$ be a part of $K$ and $A:P\times
P\rightarrow K$ a mixed monotone operator.

Assume there exists $\varphi:(0,1)\rightarrow(0,1]$ such that $\varphi
(\lambda)>\lambda$ for all $\lambda\in(0,1)$ and%
\begin{equation}
A(\lambda x,y)\geq\varphi(\lambda)A(x,\lambda y)\text{\quad for all }%
\lambda\in(0,1)\text{ and }x,y\in P\text{ linearly dependent.}
\label{eq:0711_04}%
\end{equation}

If there exists $u\in P$ such that $A(u,u)\in P$, then the following
conclusions hold:

\begin{enumerate}
\item for every $x,y\in P$, there exists $(x_{0},y_{0})\in P\times P$ a
coupled lower-upper fixed point of $A$ such that $x,y\in\lbrack x_{0},y_{0}]$;

\item $A(P\times P)\subseteq P$;

\item there exists $x^{\ast}\in P$ such that $(x^{\ast},x^{\ast})$ is the
unique coupled fixed point of $A$ in $P\times P$, $x^{\ast}$ is the unique
fixed point of $A$ in $P$ and $[x_{0},y_{0}]\overset{A}{\rightrightarrows
}x^{\ast}$ for every $(x_{0},y_{0})\in P\times P$ coupled lower-upper fixed
point of $A$.
\end{enumerate}
\end{theorem}

\begin{proof}
First, we prove that $A$ has at most a fixed point in $P$. For that, assume
$x^{\ast},y^{\ast}\in P$ be two distinct fixed points of $A$. Let%
\[
T=\left\{  \lambda>0:\lambda x^{\ast}\leq y^{\ast}\leq\lambda^{-1}x^{\ast
}\right\}
\]
and $\lambda_{\ast}=\sup T$. Obviously, $T$ is nonempty since $x^{\ast
},y^{\ast}$ are in the same part of $K$, and $\lambda_{\ast}\in T\subseteq
(0,1)$ since $K$ is Archimedean and $x^{\ast}\neq y^{\ast}$. Then, by
(\ref{eq:0711_04}) and the mixed monotonicity of $A$,%
\begin{equation}
\varphi(\lambda_{\ast})x^{\ast}=\varphi(\lambda_{\ast})A(x^{\ast},x^{\ast
})\leq A(\lambda_{\ast}x^{\ast},\lambda_{\ast}^{-1}x^{\ast})\leq A(y^{\ast
},y^{\ast})=y^{\ast}\text{,} \label{eq:0711_07}%
\end{equation}
hence $\varphi(\lambda_{\ast})x^{\ast}\leq y^{\ast}$. Due to the symmetry, one
also has $\varphi(\lambda_{\ast})y^{\ast}\leq x^{\ast}$ which shows that
$\varphi(\lambda_{\ast})\in T$, hence $\varphi(\lambda_{\ast})\leq
\lambda_{\ast}$, which contradicts the hypothesis on $\varphi$. Concluding,
$A$ has at most a fixed point in $P$.\medskip

By following the same argument as before, we have that if $(x^{\ast},y^{\ast
})$ is a coupled fixed point of $A$ in $P\times P$, then $x^{\ast}=y^{\ast}$;
the only difference from the previous argument is that (\ref{eq:0711_07}) is
replaced by:%
\[
\varphi(\lambda_{\ast})x^{\ast}=\varphi(\lambda_{\ast})A(x^{\ast},y^{\ast
})\leq A(\lambda_{\ast}x^{\ast},\lambda_{\ast}^{-1}y^{\ast})\leq A(y^{\ast
},x^{\ast})=y^{\ast}\text{.}%
\]
\medskip

The next step in our proof is to claim that $\varphi$ can be assumed to
satisfy%
\begin{equation}
\varphi(\lambda)\varphi(\mu)\leq\varphi(\lambda\mu)\quad\text{for all }%
\lambda,\mu\in(0,1) \label{eq:0711_04a}%
\end{equation}
without any loss of generality. In order to prove this, define the set%
\[
\Phi(\lambda)=\left\{  \eta\in(0,1]:A(\lambda x,y)\geq\eta A(x,\lambda
y)\text{ for all }x,y\in P\text{ linearly dependent}\right\}
\]
for every $\lambda\in(0,1)$ and consider the function $\phi:(0,1)\rightarrow
(0,1]$ given by
\[
\phi(\lambda)=\sup\Phi(\lambda)\quad(\lambda\in(0,1))\text{.}%
\]
Since $\varphi(\lambda)\in\Phi(\lambda)$ for all $\lambda\in(0,1)$, then
$\phi$ is correctly defined and $\phi(\lambda)\geq\varphi(\lambda)>\lambda$
for all $\lambda\in(0,1)$. Also, $\phi(\lambda)\in$ $\Phi(\lambda)$ since $K$
is Archimedean, hence
\[
A(\lambda x,y)\geq\phi(\lambda)A(x,\lambda y)\text{\quad for all }\lambda
\in(0,1)\text{ and }x,y\in P\text{ linearly dependent.}%
\]
Moreover, for all $\lambda,\mu\in(0,1)$ and $x,y\in P$ linearly dependent,
\[
A(\lambda\mu x,y)\geq\phi(\lambda)A(\mu x,\lambda y)\geq\phi(\lambda)\phi
(\mu)A(x,\lambda\mu y)
\]
which shows that $\phi(\lambda)\phi(\mu)\in\Phi(\lambda\mu)$, hence
$\phi(\lambda)\phi(\mu)\leq\phi(\lambda\mu)$. It is clear now that by
replacing $\varphi$ with $\phi$, we obtain the desired property
(\ref{eq:0711_04a}).\bigskip

Next, since $u$ and $A(u,u)$ are in the same part of $K$, there exists
$\lambda_{0}\in(0,1)$ such that
\begin{equation}
\lambda_{0}u\leq A(u,u)\leq\lambda_{0}^{-1}u. \label{eq:0711_05}%
\end{equation}
Also, $\lim_{n\rightarrow\infty}\left(  \frac{\varphi(\lambda_{0})}%
{\lambda_{0}}\right)  ^{n}=\infty$ since $\frac{\varphi(\lambda_{0})}%
{\lambda_{0}}>1$, hence there exists $k_{0}\in\mathbb{N}$ such that $\left(
\frac{\varphi(\lambda_{0})}{\lambda_{0}}\right)  ^{n}\geq\lambda_{0}^{-1}$ for
all $n\geq k_{0}$, i.e.,%
\begin{equation}
\lambda_{0}^{n}\leq\left(  \varphi(\lambda_{0})\right)  ^{n}\lambda
_{0}\text{\quad for all }n\geq k_{0}\text{.} \label{eq:0711_06}%
\end{equation}

Now, consider arbitrary $x,y\in P$. Since $x$, $y$ are in the same part of the
cone with $u$, there exists $n_{0}\geq k_{0}$ large enough such that
$\lambda_{0}^{n_{0}}u\leq x\leq\lambda_{0}^{-n_{0}}u$ and $\lambda_{0}^{n_{0}%
}u\leq y\leq\lambda_{0}^{-n_{0}}u$. Let $x_{0}=\lambda_{0}^{n_{0}}u$ and
$y_{0}=\lambda_{0}^{-n_{0}}u$. Clearly, $x_{0},y_{0}\in P$, $x_{0}\leq y_{0}$
and $x,y\in\lbrack x_{0},y_{0}]$. By successively applying (\ref{eq:0711_04}%
)--(\ref{eq:0711_06}) several times and using the mixed monotonicity of $A$,
we have that%
\begin{align*}
x_{0} &  =\lambda_{0}^{n_{0}}u\leq\left(  \varphi(\lambda_{0})\right)
^{n_{0}}\lambda_{0}u\leq\varphi(\lambda_{0}^{n_{0}})A(u,u)\leq A(\lambda
_{0}^{n_{0}}u,\lambda_{0}^{-n_{0}}u)=A(x_{0},y_{0})\\
&  \leq A(x,y)\leq A(y_{0},x_{0})=A(\lambda_{0}^{-n_{0}}u,\lambda_{0}^{n_{0}%
}u)\leq\left(  \varphi(\lambda_{0}^{n_{0}})\right)  ^{-1}A(u,u)\\
&  \leq\left(  \lambda_{0}\left(  \varphi(\lambda_{0})\right)  ^{n_{0}%
}\right)  ^{-1}u\leq\lambda_{0}^{-n_{0}}u=y_{0},
\end{align*}
which shows that $(x_{0},y_{0})$ is a coupled lower-upper fixed point of $A$
and
\[
A(x,y)\in\lbrack x_{0},y_{0}]\subseteq P,
\]
hence $A(P\times P)\subseteq P$.\bigskip

Now, let $(x_{0},y_{0})$ be any coupled lower-upper fixed point of $A$. In
order to conclude the proof, it is enough to show that there exists $x^{\ast
}\in\lbrack x_{0},y_{0}]$ such that $(x_{0},y_{0}%
)\overset{A}{\rightrightarrows}x^{\ast}$, and the conclusion will follow from
Theorem \ref{th:0711_08}. In order to achieve this, let $(x_{n}),(y_{n})$ be
defined as in (\ref{eq:0711_00}), hence by Lemma \ref{th:0711_01},
\begin{equation}
x_{0}\leq x_{1}\leq\ldots\leq x_{n}\leq\ldots\leq y_{m}\leq\ldots\leq
y_{1}\leq y_{0} \label{eq:0711_09}%
\end{equation}
We break the proof in several steps.

First, it is clear that if $x_{k}=y_{k}$ for some $k$, then $\sup x_{n}=\inf
y_{n}=x_{k}$ and the proof is complete, hence one can assume that $x_{n}\neq
y_{n}$ for all $n\in\mathbb{N}$.

Next, let the sequence $\left(  \lambda_{n}\right)  $ be defined by
$\lambda_{n+1}=\varphi(\lambda_{n})$ for all $n\in\mathbb{N}$, where
$\lambda_{0}\in(0,1)$ is such that $x_{0}\geq\lambda_{0}y_{0}$ ($\lambda_{0}$
exists, since $x_{0},y_{0}$ are in the same part of $K$). We show by induction
that, for all $n\in\mathbb{N}$,%
\begin{equation}
\lambda_{n}\text{ is correctly defined,\quad}\lambda_{n}\in(0,1)\text{,\quad
}x_{n}\geq\lambda_{n}y_{n}\text{.}\label{eq:0711_10}%
\end{equation}
Clearly, these are satisfied for $n=0$. Now, assume these properties are true
for $n$. Then $\lambda_{n+1}=\varphi(\lambda_{n})\in(0,1]$ is correctly
defined (since $\lambda_{n}\in(0,1)$) and, by (\ref{eq:0711_04}),%
\[
x_{n+1}=A(x_{n},y_{n})\geq A(\lambda_{n}y_{n},y_{n})\geq\varphi(\lambda
_{n})A(y_{n},\lambda_{n}y_{n})\geq\varphi(\lambda_{n})A(y_{n},x_{n}%
)=\lambda_{n+1}y_{n+1}%
\]
hence $x_{n+1}\geq\lambda_{n+1}y_{n+1}$. Since $x_{n+1}\neq y_{n+1}$, it also
follows from here that $\lambda_{n+1}\neq1$, hence $\lambda_{n+1}\in(0,1)$,
which concludes the inductive proof.

Note also that $\lambda_{n}<\varphi(\lambda_{n})=\lambda_{n+1}$ for all
$n\in\mathbb{N}$. Following from here, we conclude that the sequence
$(\lambda_{n})$ is increasing, hence convergent to some $\lambda_{\ast}%
\in(0,1]$; we prove that $\lambda_{\ast}=1$. Assume that $\lambda_{\ast}\neq
1$. Clearly, $\lambda_{n}<\lambda_{\ast}$ for all $n\in\mathbb{N}$. Then, by
(\ref{eq:0711_04a}),%
\[
\lambda_{n+1}=\varphi(\lambda_{n})=\varphi\left(  \lambda_{\ast}\cdot
\frac{\lambda_{n}}{\lambda_{\ast}}\right)  \geq\varphi(\lambda_{\ast}%
)\varphi\left(  \frac{\lambda_{n}}{\lambda_{\ast}}\right)  >\varphi
(\lambda_{\ast})\frac{\lambda_{n}}{\lambda_{\ast}}\quad\text{for all }%
n\in\mathbb{N}%
\]
and by taking $n\rightarrow\infty$, we obtain that $\lambda_{\ast}\geq
\varphi(\lambda_{\ast})$, which is a contradiction. Concluding,
\begin{equation}
\lim_{n\rightarrow\infty}\lambda_{n}=1\text{.}\label{eq:0711_08}%
\end{equation}

We claim now that $(x_{n})$ is upper self-bounded. Indeed, let $\mu>1$ and, by
(\ref{eq:0711_08}), let $k\in\mathbb{N}$ such that $\lambda_{k}\geq\mu^{-1}$.
Then, by (\ref{eq:0711_09}) and (\ref{eq:0711_10}),%
\[
x_{n}\leq\mu\lambda_{k}x_{n}\leq\mu\lambda_{k}y_{k}\leq\mu x_{k}\quad\text{for
all }n\in\mathbb{N}\text{,}%
\]
which proves our claim.

Next, we use that $K$ is self-complete, hence there exists $x^{\ast}=\sup
x_{n}$.

Finally, we show that $x^{\ast}=\inf y_{n}$. Indeed, $x^{\ast}\leq y_{n}$ for
all $n\in\mathbb{N}$ (by (\ref{eq:0711_09})). Also, if $x\in X$ such that
$x\leq y_{n}$ for all $n\in\mathbb{N}$, then
\[
x\leq y_{n}\leq\frac{x_{n}}{\lambda_{n}}\leq\frac{x^{\ast}}{\lambda_{n}%
}\text{\quad for all }n\in\mathbb{N}%
\]
hence $x\leq x^{\ast}$, by (\ref{eq:0711_08}) and using that $K$ is
Archimedean. Concluding, $(x_{0},y_{0})\overset{A}{\rightrightarrows}x^{\ast}$
and the proof is now complete.
\end{proof}

\subsection*{Acknowledgement}

This work was supported by the Sectoral Operational Programme Human Resources
Development 2007-2013 of the Romanian Ministry of Labor, Family and Social
Protection through the Financial Agreement POSDRU/89/1.5/S/62557.




\begin{thebibliography}{10}

\bibitem{Chen1993}
Y.~Z. Chen, \textit{{T}hompson's metric and mixed monotone operators}, J. Math.
  Anal. Appl., \textbf{177}(1993), no.~1, 31--37.

\bibitem{GnanaBhaskar2006}
T.~Gnana~Bhaskar and V.~Lakshmikantham, \textit{Fixed point theorems in
  partially ordered metric spaces and applications}, Nonlinear Anal.,
  \textbf{65}(2006), no.~7, 1379--1393.

\bibitem{Guo2004}
D.~Guo, Y.~J. Cho, and J.~Zhu, \textit{Partial ordering methods in nonlinear
  problems}, Nova Science Publishers Inc., Hauppauge, NY, 2004.

\bibitem{Guo1988}
D.~J. Guo, \textit{Fixed points of mixed monotone operators with applications},
  Appl. Anal., \textbf{31}(1988), no.~3, 215--224.

\bibitem{Guo1987}
D.~J. Guo and V.~Lakshmikantham, \textit{Coupled fixed points of nonlinear
  operators with applications}, Nonlinear Anal., \textbf{11}(1987), no.~5,
  623--632.

\bibitem{Harjani2011a}
J.~Harjani, B.~L\'{o}pez, and K.~Sadarangani, \textit{Fixed point theorems for
  mixed monotone operators and applications to integral equations}, Nonlinear
  Anal., \textbf{74}(2011), no.~5, 1749--1760.

\bibitem{Jameson1968}
G.~Jameson, \textit{Allied subsets of topological groups and linear spaces},
  Proc. London Math. Soc. (3), \textbf{18}(1968), 653--690.

\bibitem{Karapinar2010}
E.~Karap{\i}nar, \textit{Coupled fixed point theorems for nonlinear
  contractions in cone metric spaces}, Comput. Math. Appl., \textbf{59}(2010),
  no.~12, 3656--3668.

\bibitem{Kurpel1968}
N.~S. Kurpel$'$, \textit{Some methods of constructing two-sided approximations
  to solutions of operator equations}, in \textit{Problems in the {T}heory and
  {H}istory of {D}ifferential {E}quations ({R}ussian)}, ``Naukova Dumka'', Kiev
  (1968), 131--146.

\bibitem{Kurpel1980}
N.~S. Kurpel$'$ and B.~A. {\v{S}}uvar, \textit{Two-sided operator inequalities
  and their applications (in Russian)}, Naukova Dumka, Kiev (1980).

\bibitem{Lakshmikantham2009}
V.~Lakshmikantham and L.~{\'C}iri{\'c}, \textit{Coupled fixed point theorems
  for nonlinear contractions in partially ordered metric spaces}, Nonlinear
  Anal., \textbf{70}(2009), no.~12, 4341--4349.

\bibitem{Luong2011}
N.~V. Luong and N.~X. Thuan, \textit{Coupled fixed points in partially ordered
  metric spaces and application}, Nonlinear Anal., \textbf{74}(2011), no.~3,
  983--992.

\bibitem{Moore1981}
J.~Moore, \textit{Existence of multiple quasifixed points of mixed monotone
  operators by iterative techniques}, Appl. Math. Comput., \textbf{9}(1981),
  no.~2, 135--141.

\bibitem{Olatinwo2011}
M.~Olatinwo, \textit{Coupled fixed point theorems in cone metric spaces}, Ann.
  Univ. Ferrara Sez. VII Sci. Mat., \textbf{57}(2011), 173--180.

\bibitem{Opoitsev1975}
V.~I. Opo{\u\i}tsev, \textit{Heterogeneous and combined-concave operators (in
  russian)}, Sibirsk. Mat. \v Z., \textbf{16}(1975), no.~4, 781--792.

\bibitem{Opoitsev1978a}
V.~I. Opo{\u\i}tsev, \textit{Generalization of the theory of monotone and
  concave operators (in {R}ussian)}, Trudy Moskov. Mat. Obshch.,
  \textbf{36}(1978), 237--273, {E}nglish translation in {T}rans. {M}oscow
  {M}ath. {S}oc. (1979), no. 2, 243--279.

\bibitem{Picard1890}
E.~Picard, \textit{M\'{e}moire sur la theorie des \'{e}quations aux
  d\'{e}riv\'{e}es partielles et la m\'{e}thode des approximations
  successives}, J. de Math., \textbf{6}(1890), 145--210.

\bibitem{Picard1893}
E.~Picard, \textit{Sur l'application des m\'{e}thodes d'approximations
  successives \`{a} l'\'{e}tude de certaines \'{e}quations diff\'{e}rentielles
  ordinaires}, J. de Math., \textbf{9}(1893), 217--272.

\bibitem{Picard1900}
E.~Picard, \textit{Sur un exemple d'approximations successives divergentes},
  Bull. Soc. Math. France, \textbf{28}(1900), 137--143.

\bibitem{Rus2011}
M.-D. Rus, \textit{Fixed point theorems for generalized contractions in
  partially ordered metric spaces with semi-monotone metric}, Nonlinear Anal.,
  \textbf{74}(2011), no.~5, 1804--1813.

\bibitem{Sabetghadam2009}
F.~Sabetghadam, H.~P. Masiha, and A.~H. Sanatpour, \textit{Some coupled fixed
  point theorems in cone metric spaces}, Fixed Point Theory Appl.,
  \textbf{2009}(2009), Art. ID 125426, 8 pages.

\bibitem{Samet2010}
B.~Samet, \textit{Coupled fixed point theorems for a generalized
  {M}eir--{K}eeler contraction in partially ordered metric spaces}, Nonlinear
  Anal., \textbf{72}(2010), no.~12, 4508--4517.

\bibitem{Shatanawi2010a}
W.~Shatanawi, \textit{Partially ordered cone metric spaces and coupled fixed
  point results}, Comput. Math. Appl., \textbf{60}(2010), no.~8, 2508--2515.

\bibitem{Shatanawi2010}
W.~Shatanawi, \textit{Some common coupled fixed point results in cone metric
  spaces}, Int. J. Math. Anal. (Ruse), \textbf{4}(2010), no. 45-48, 2381--2388.

\bibitem{Thompson1963}
A.~C. Thompson, \textit{On certain contraction mappings in a partially ordered
  vector space}, Proc. Amer. Math. Soc., \textbf{14}(1963), 438--443.

\bibitem{Wu2006}
Y.~Wu and Z.~Liang, \textit{Existence and uniqueness of fixed points for mixed
  monotone operators with applications}, Nonlinear Anal., \textbf{65}(2006),
  no.~10, 1913--1924.

\bibitem{Xu2005a}
B.~Xu and R.~Yuan, \textit{On the positive almost periodic type solutions for
  some nonlinear delay integral equations}, J. Math. Anal. Appl.,
  \textbf{304}(2005), no.~1, 249--268.

\end{thebibliography}
\end{document}